\documentclass[sn-mathphys-num]{sn-jnl}% Math and Physical Sciences Numbered Reference Style 
\usepackage{graphicx}%
\usepackage{multirow}%
\usepackage{amsmath,amssymb,amsfonts}%
\usepackage{amsthm}%
\usepackage{mathrsfs}%
\usepackage[title]{appendix}%
\usepackage{xcolor}%
\usepackage{textcomp}%
\usepackage{manyfoot}%
\usepackage{booktabs}%
\usepackage{algorithm}%
\usepackage{algorithmicx}%
\usepackage{algpseudocode}%
\usepackage{listings}%
\usepackage{natbib}
\usepackage{mathtools}

\DeclarePairedDelimiter \ceil{ \lceil}{ \rceil}

\newtheorem{theorem}{Theorem}[section]% meant for sectionwise numbers
\newtheorem{prop}[theorem]{Proposition}% 
\newtheorem{definition}{Definition}[section]%
\newtheorem{lemma}{Lemma}[section]%
\newtheorem{coro}{Corollary}[section]
\newtheorem{obs}{Observation}[section]

\newtheorem{construction}{Construction}
\begin{document}

\title[Article Title]{Efficient $k$-limited Dominating Broadcasts in Product Graphs}

%%=============================================================%%
%% GivenName	-> \fnm{Joergen W.}
%% Particle	-> \spfx{van der} -> surname prefix
%% FamilyName	-> \sur{Ploeg}
%% Suffix	-> \sfx{IV}
%% \author*[1,2]{\fnm{Joergen W.} \spfx{van der} \sur{Ploeg} 
%%  \sfx{IV}}\email{iauthor@gmail.com}
%%=============================================================%%

\author*[1]{\fnm{Bharadwaj} \sur{ }}\email{bharadwaj.217ma004@nitk.edu.in}
% \equalcont{These authors contributed equally to this work.}
\author[2]{\fnm{A.Senthil} \sur{Thilak}}\email{thilak@nitk.edu.in}
% \equalcont{These authors contributed equally to this work.}

%\author[1,2]{\fnm{Third} \sur{Author}}\email{iiiauthor@gmail.com}
%\equalcont{These authors contributed equally to this work.}

\affil*[1,2]{\orgdiv{Department of Mathematical and Computational Sciences}, \orgname{National Institute of Technology Karnataka, Surathkal}, \orgaddress{\street{NH 66, Srinivasnagar 
Surathkal}, \city{Mangalore}, \postcode{575025}, \state{Karnataka}, \country{India}}}

%%==================================%%
%% Sample for unstructured abstract %%
%%==================================%%

\abstract{In a graph $ G $, a subset of vertices $ S $ is called an efficient dominating set (EDS) if every vertex in the graph is uniquely dominated by exactly one vertex in $ S $. A graph is said to be efficiently dominatable if it contains an EDS. Additionally, a function $ f: V(G) \rightarrow \{0, 1, 2, \dots, k\} $ is termed a $ k $-limited dominating broadcast if, for every vertex $ u \in V(G) $, there exists a vertex $ v $, with $ f(v) \geq 1$ such that $ d(u, v) \leq f(v) $. A vertex $u$ is said to be dominated by a vertex $v$. In this work, we unify these two concepts to explore the notion of efficient $k$-limited broadcast domination in graphs. A $ k $-limited dominating broadcast $f$ is called an efficient $k$-limited dominating broadcast ($k$-$ELDB$) if each vertex in the graph is dominated exactly once. The minimum value of $k$ for which the given graph $G$ has $k$-$ELDB$ is defined as $mcr(G)$. We prove determining $mcr(G)$ is NP-Complete for general graphs and explore the $mcr(G)$ values and other related parameters on standard graphs and their products.}


%%\pacs[JEL Classification]{D8, H51}
%\pacs[MSC Classification]{15A09, 15A23, 15A24.}
%%\pacs[MSC Classification]{35A01, 65L10, 65L12, 65L20, 65L70}

\maketitle

\section{Introduction}

---------------------------------------------------

\noindent Given a graph $G = (V, E)$, a set $S \subseteq V(G)$ is called a \textit{dominating set} if each vertex $v \in V(G)$ is either in $S$ or has at least one neighbor in $S$. The size of the smallest dominating set of $G$ is referred to as the \textit{domination number of $G$} and is denoted by $\gamma(G)$. The \textit{open neighborhood of a vertex $v$}, denoted by $N(v)$, is the set of all vertices adjacent to $v$. The \textit{closed neighborhood of $v$}, denoted by $N[v]$, is defined as $N(v) \cup \{v\}$. A set $S \subseteq V(G)$ is an \textit{efficient dominating set} (\textit{EDS}) of $G$ if $|N[v] \cap S| = 1$ for all $v \in V(G)$. In other words, $S$ is an \textit{EDS} if each vertex $v \in V(G)$ is dominated by exactly one vertex (including itself) in $S$. Not every graph possesses an \textit{EDS}. If a graph $G$ has an \textit{EDS}, then it is said to be \textit{efficiently dominatable}.

The \textit{distance} between a pair of vertices $u$ and $v$ is the length of the shortest path between $u$ and $v$, denoted by $d(u,v)$. A set $S\subseteq V(G)$ is a \textit{2-packing} if for each pair $u,v \in S$, $N[u]\cap N[v] = \emptyset$. If $S$ is a \textit{2-packing}, then $d(u,v) \geq 3,$ for all $u,v \in S$. Thus, a dominating set is an \textit{EDS} if and only if it is a $2$-packing. The \textit{influence} of a set $S\subseteq  V(G)$ is denoted by $ I(S) $ and is the number of vertices dominated by $ S $ (inclusive of vertices in $S$). If $S$ is a $2$-packing, then $ I(S)= \sum_{v\in S}  [1+ \deg(v)].$ The maximum influence of a $2$-packing  of $G$ is called the \textit{efficient domination number} of $G$ and is denoted by $F(G)$. That is, $F(G) = \max\{I(S):~S \text{ is a 2-packing}\} = \max \left\{\sum_{v \in S}(1 + \operatorname{deg} v): S \subseteq V(G)\right.$ and $\left.|N[x] \cap S| \leq 1,  \forall~ x \in V(G)\right\}$. Clearly, $0 \le F(G) \le |V(G)|$ and $G$ is efficiently dominatable if and only if $F(G) = |V(G)|$.  A $2$-packing with influence $F(G)$ is called an \textit{$F(G)$-set}.

The concept of efficient domination is found in the literature under different names such as \textit{perfect codes}  or \textit{perfect $1$-codes} \cite{biggs1973}, \textit{independent perfect domination} \cite{yen1996},  \textit{perfect $ 1 $-domination} \cite{livingston1990}, and \textit{efficient domination} \cite{bange1988}. In this paper, we use the terminology \textit{efficient domination} introduced by Bange et al.~\cite{bange1988}. The problem of determining whether $ F(G) = |V(G)| $ is $\mathcal{NP} $-complete on arbitrary graphs \cite{bange1988} as well as on some special/restricted classes of graphs such as bipartite graphs, chordal graphs, planar graphs of degree at most three, etc. \cite{haynes1998}. In contrast, it is solvable in polynomial time for trees \cite{bange1988}. Goddard et al.~\cite{goddard2000} have derived bounds on the efficient domination number of arbitrary graphs and trees. Efficient domination has also been studied in various special classes of graphs such as chordal bipartite graphs \cite{tang2002}, the strong product of arbitrary graphs \cite{abay2009}, and the Cartesian product of cycles \cite{chelvam2011}. Hereditary efficiently dominatable graphs were defined and studied in \cite{milanic2013} and \cite{barbosa2016}. The concept of perfect codes or efficient domination finds wide applications in coding theory, resource allocation in computer networks, and more \cite{va2001,livingston1988}.

In 2001, \citeauthor{erwin2001cost} introduced a variant of domination termed \textit{broadcast domination}, which models multiple broadcast stations with associated transmission powers capable of disseminating messages to locations beyond a unit distance. A \textit{broadcast} on a graph $G$ is a function $f : V(G) \rightarrow \lbrace 0, 1, 2, \dots, \text{diam}(G)\rbrace$ with $f(v) \leq \text{ecc}(v)$ for all $v \in V(G)$. A vertex $v$ is a \textit{broadcasting vertex} if $f(v) \geq 1$, and the set of all broadcasting vertices is denoted by $V_{f}^{+}(G)$ or simply $V_{f}^{+}$. The set of vertices with $f(v)=0$ is denoted by $V_{f}^{0}(G)$. A vertex $u \in V(G)$ is said to be \textit{f-dominated} if there is a vertex $v \in V_{f}^{+}$ such that $d(v,u) \leq f(v)$. Then, vertex $u$ is said to \textit{hear the broadcast} from vertex $v$. The set of vertices that vertex $u$ can hear is defined as $H(u)=\{v \in V_{f}^{+} \mid  d(u,v) \leq f(v)\}$. A vertex $u$ is \textit{overdominated} if there is a broadcasting vertex $v$ corresponding to a broadcast $f$ such that $d(u,v) < f(v)$. For each vertex $v \in V_{f}^{+}$, the \textit{closed f-neighborhood} of $v$ is the set $\lbrace u \in V(G) : d(u,v) \leq f(v) \rbrace$ and is denoted by $N_{f}[v]$. A broadcast $f$ on a graph $G$ is said to be a \textit{dominating broadcast} on $G$ if $\bigcup_{v \in V_{f}^{+}} N_{f}[v] = V(G)$. The \textit{cost} of a dominating broadcast $f$ is defined as $\sigma(f) = \sum_{v \in V_{f}^{+}} f(v)$. The \textit{broadcast domination number} of $G$ ($\gamma_b(G)$ ) is $\min \lbrace \sigma(f) \mid f \text{ is a dominating broadcast on } G \rbrace$.

It was observed in \cite{erwin2001cost} that $\gamma_b(G) \leq \min \{\operatorname{rad}(G), \gamma(G)\}$. The graphs for which $\gamma_b(G) = \operatorname{rad}(G)$ are called \textit{radial graphs}. \cite{heggernes2006} proved $\gamma_{b}(G)$ can be computed in polynomial time for any graph $G$. 

In \cite{dunbar2006broadcasts}, \textit{efficient broadcast} was defined as follows: A broadcast $f$ is called \textit{efficient} if every vertex hears from exactly one broadcasting vertex; that is, for every vertex $v \in V(G)$, $|H(v)| =1$. A dominating broadcast $f$ is considered \textit{efficient} if every vertex is either a broadcasting vertex or $f$-dominated by exactly one broadcasting vertex. It was proved that every graph $G$ has an optimal and efficient dominating broadcast. In the same paper, a limited version of broadcast domination was proposed. A function $f: V(G) \rightarrow \{0, 1, 2, \dots, k\}$ is termed a $k$-limited dominating broadcast if, for every vertex $u \in V(G)$, there exists a vertex $v$, with $f(v) \geq 1$ such that $d(u, v) \leq f(v)$. The cost of a $k$-limited dominating broadcast $f$ is $\omega(f)=\sum_{u \in V(G)} f(u)=\sum_{v \in V_{f}^{+}} f(v)$. The $k$-limited dominating broadcast number of $G$ is defined as:
\[
\gamma_{B_{k}}(G)=\min \{\omega(f) \mid f \text{ is a $k$-limited dominating broadcast on } G\}.
\]
In \cite{caceres2018}, bounds on $\gamma_{B_{k}}(G)$ were discussed, and it was proved that determining $\gamma_{B_{k}}(G)$ for general graphs is NP-Complete. The broadcast domination chapter in the book \cite{haynes2021structures} summarizes these results on broadcast domination and its variants. Motivated by the open problem posed in this chapter \textit{What is the smallest value of $k$ for which a graph $G$ has an efficient $k$-limited dominating broadcast?} we investigate this problem for certain standard graphs and their graph products and explore the computational hardness of the same for general graphs.

\begin{definition}
    An \textit{efficient $k$-limited dominating broadcast} ($k$-ELDB) $f$ is a $k$-limited dominating broadcast in which every vertex is $f$-dominated by exactly one vertex.
\end{definition}

For convenience, we call an efficient $k$-limited dominating broadcast a \textit{$k$-ELDB}. In this paper, we call a graph $G$ \textit{$k$-efficiently dominatable} if it possesses a $k$-ELDB. Note that any graph $G$ is $k$-efficiently dominatable for $k=\operatorname{rad}(G)$. If a graph $G$ is $k$-efficiently dominatable, then it is $l$-efficiently dominatable for any $l > k$. Thus, the study of the minimum value of $k$ for which a graph is $k$-efficiently dominatable is significant, and we call it the \textit{minimum covering radius} of a graph, denoted by $\text{mcr}(G)$.

\begin{definition}
    The \textit{minimum covering radius} of a graph is the minimum value of $k$ for which the graph is $k$-efficiently dominatable.
\end{definition}

It can be noted that for any graph $G$, $1 \leq \text{mcr}(G) \leq \operatorname{rad}(G)$. We define the \textit{optimal cost} of $k$-$ELDB$ as follows:

\begin{definition}
    The \textit{$k$-efficient broadcast domination number} of a graph $G$, denoted by $\gamma_{ebk}(G)$, is defined as:
    \[
    \gamma_{ebk}(G) = \min \left\{ \sum_{v \in V_{f}^{+}} f(v) \mid  f \text{ is a $k$-ELDB on $G$} \right\}.
    \]
\end{definition}

\begin{obs}
    $\gamma_{eb1}(G) \leq \gamma_{eb2}(G) \leq \dots \leq \gamma_{ebk}(G)$, where $k = \text{mcr}(G)$.
\end{obs}

If $k$-ELDB does not exist on the graph $G$ for any $k < \text{mcr}(G)$, we find the maximum number of vertices that a $k$-ELDB can dominate.

\begin{definition}
    The \textit{$k$-efficient covering number}, denoted by $F_k(G)$, is defined as:
    \[
    F_k(G) = \max \left\{\left| \bigcup_{v \in V_{f}^{+}} N_{f}[v] \right| \mid f \text{ is a $k$-ELDB on $G$} \right\}.
    \]
\end{definition}

\begin{obs}
    $F(G) \leq F_2(G) \leq \dots \leq F_{\operatorname{rad}(G)}=|V(G)|$
\end{obs}

The concept of an efficient $k$-limited dominating broadcast is particularly useful in applications where redundancy minimization and precise control are necessary. In the design of urban surveillance systems, for example, efficient broadcast can help position cameras such that each area is monitored by exactly one camera, optimizing resource allocation. In distributed computing, an efficient broadcast can be used to identify optimal node placements to achieve efficient and fault-tolerant communication protocols.

\section{Efficient $k$-limited dominating broadcast on different classes of graphs}

This section aims to investigate $k$-$ELDB$ on specific graph classes: Paths ($P_n$), Cycles ($C_n$), Subdivision of stars ($S_i(K_{1,n}$)), Complete graphs ($K_m$), lexicographic, strong and cartesian product of graphs. It can be noted that if $mcr(G)=1$, then $G$ is efficiently dominatable. 

\subsection{Results on some well-known graphs}

\begin{theorem}
    $ \frac{2n}{1+\Delta^2} \leq \gamma_{eb2}(G) \leq \frac{n}{1+\delta(G)}$
\end{theorem}

\begin{proof}
    A vertex receiving cost $2$ can dominate at most $1+\Delta +\Delta(\Delta-1) = 1+\Delta^2$ number of vertices. On the other hand, a broadcasting vertex receiving cost $1$ dominates at least ${1+\delta(G)}$ number of vertices. 
\end{proof}

\begin{theorem}
    $P_n$ is $1$-efficiently dominatable. $\gamma_{eb1}(P_n)= \gamma_{ebk}(P_n) =  \ceil* {\frac{n}{3}}, \forall k \geq 2$.
\end{theorem}
\begin{theorem} \label{Cm}

$ mcr(C_n)=\begin{cases}
{1} &\mbox{if $n \equiv 0$ $(mod 3) $  }\\
{2} &\mbox{if $n \equiv 1$ or $2$ $(mod 3)$, $n \neq 7$}\\
{3} &\mbox{if $n=7$}
\end{cases} $\\

$\gamma_{ebk}(C_n)= \ceil*{\frac{n}{3}} $, where $k=mcr(C_n)$.

\end{theorem}

\begin{proof}

A vertex assigned a cost of $1$ and $2$ dominates three and five vertices (including itself), respectively. It is evident that in any $\gamma_{ebk}$-broadcast of $C_n$, minimizing the number of vertices assigned a cost of $2$ is essential to reduce the total cost.

\noindent \textbf{Case(i):} $n \equiv 0 \pmod{3}$\\
That is, $n = 3k$, a vertex with a cost of $1$ can dominate three vertices. Hence, $k$ vertices, each assigned a cost of $1$ and positioned at a distance of $3$, constitute a $1$-$ELDB$ of $C_n$.

\noindent \textbf{Case(ii):} $n \equiv 1 \pmod{3}$\\
That is, $n = 3k + 1$, we focus on $n \geq 10$ as the result holds for $n = 4$. Since $n = 2(5) + (k - 3)(3)$, $k - 3$ vertices are assigned a cost of $1$, and two vertices are assigned a cost of $2$. Thus, $k + 1$ vertices form the set ${V_f}^+$ for the $\gamma_{ebk}$-broadcast.

\noindent \textbf{Case(iii):}  $n \equiv 2 \pmod{3}$\\
That is, $n = 3k + 2$, then $n = 1(5) + (k - 1)(3)$ implies that $k - 1$ vertices are assigned a cost of $1$, and one vertex is assigned a cost of $2$. Therefore, $k + 1$ vertices form the set ${V_f}^+$ for the $\gamma_{ebk}$-broadcast.

\end{proof}
An edge subdivision is the insertion of a new vertex $v_j$ in the middle of an existing edge $e=v_iv_k$ accompanied by the joining of the original edge endpoints with the new vertex to form new edges $e_1 =v_i v_j$ and $e_2=v_jv_k$.

Let $(S_i(K_{1,n-1}))$ be $i^{th}$ subdivision of a star graph. Let $v_0, v_1, v_2, \dots v_i, v_{i+1}$ be the vertices on path from central vertex $v_0$ to a leaf vertex $v_{i+1}$. There are $n-1$ such paths. 
\begin{theorem}
    $mcr(S_i(K_{1,n-1}))=1$ and \\
    $\gamma_{eb1}(S_i(K_{1,n-1})) =\begin{cases}
{\ceil*{\frac{i}{3}}}(n-1) &\mbox{ if $i \equiv 1 $ $(mod 3)$}\\
{\ceil*{\frac{i}{3}}}(n-1)+1 &\mbox{ otherwise  }

\end{cases} $\\
\end{theorem}

\begin{proof}
    We give $1$-$ELDB$ and prove it is optimal.\\
     Let $v_0$ be the central vertex of $S_i(K_{1,n-1})$. $P_{v_ov_{i+1}}$ denote the path from $v_0, v_1, v_2, \dots v_i, v_{i+1}$. 
    
\noindent    \textbf{Case(i):} $i \equiv 1 \pmod{3}$\\
    That is, $i = 3k+1$ for $k\geq 0$, Choose vertices $v_2, v_5, \dots v_{i+1}$ of all such paths except one. Choose $v_1, v_4, \dots v_{i}$ for that remaining path. Since to dominate each of these paths, we need at least a cost $k+1$, and to dominate the whole graph, we need at least $k(n-1)$ vertices. \\
\textbf{Case(ii):} $i \equiv 0$ or $2 \pmod{3}$\\
That is, $i=3k$ or $3k+2$. Choose the central vertex $v_0$ and vertices  $v_3,v_6, \dots v_{i+1}$ of all the paths. If $v_1$ of any of these paths belong to ${V_{f}}^+$, then at least $k+1$ vertices need to dominate each of the other paths. Therefore $v_1 \notin {V_{f}}^+$. By a similar argument as in case(i), at least $k$ vertices are required to dominate each path $P(v_1,v_{3k+1})$.  And since $v_1 \notin {V_{f}}^+$, theorem holds.  

\end{proof}
\subsection{Results on Product graphs}

\subsubsection{Lexicographic Product of graphs}

\begin{definition} \cite{imrich2000}\label{def_lex}
    The lexicographic product of two graphs $G$ and $H$ is denoted by $G \cdot H$, whose vertex set is $V(G) \times V(H)$, in which two vertices $(u,v)$ and $(u', v')$ are adjacent if 
\begin{enumerate}
    \item  $uu' \in E(G)$, or 
    \item  $u=u'$ and $vv' \in E(H)$.
\end{enumerate}

\end{definition}

Throughout this section, we assume  $V(G)=\{u_1, u_2, \dots u_m\}$ and $ V(H) = \{v_1, v_2, \dots v_n\}$ be vertices of $H$, unless specified otherwise. 

The graphs $G$ and $H$ are called the \textit{factors} of $G\cdot H$. For $v\in V(H)$, the induced subgraph $G^{(v)}$ of $G\cdot H$, defined as $G^{(v)} = <\{(u,v)\in V(G\cdot H):u\in V(G)\}>$ is called the \textit{$G$-layer with respect to $v$} in $G\cdot H$.  Analogously, for $u\in V(G)$, the induced subgraph $H^{(u)}$ of $G\cdot H$, defined as $H^{(u)} = <\{(u,v)\in V(G\cdot H):v\in V(H)\}>$ is called the \textit{$H$-layer with respect to $u$} in $G\cdot H$. \\

Since $d((u_i, v_j),(u_i,x)) \leq 2$, we have the following observation.
\begin{obs} \label{obs_lex}
   $ \forall u_i \in V(G)$, there exists at most one broadcasting vertex $v$ in $V(H^{u_i})$.
\end{obs}

\begin{prop}\label{prop}

Let $f$ be $k$-$ELDB$ of $G \cdot H$ such that $f(u_i, v_j)=1$, then $\operatorname{rad}(H)=1$ with central vertex $v_j$. 

\end{prop}

\begin{proof}
    Assume the contrary. Let $v_j,v_k\in V(H)$ such that $v_jv_k\notin E(H)$. By observation \ref{obs_lex}, $(u_i,v_k)$ has to be dominated by some vertex of $H^{(u_p)}$, where $u_p \in N(u_i)$. It can be noted that all vertices of $H^{u_p}$, where $u_p \in N(u_i)$ are dominated by $(u_i, v_j)$ and hence $(u_i,v_k)$ is left undominated by $f$ contradiction to $f$ being dominating broadcast.

\end{proof}
% It can be noted that a similar result as above was proved for efficient optimal dominating broadcast in \cite{jishnu}, where the condition of efficient broadcast being optimal is redundant.
\begin{theorem}
    $mcr(G \cdot H) \geq mcr(G)$
\begin{proof}
Let $ mcr(G) = k $. Assume, for contradiction, that $ mcr(G \cdot H) = p < k $. Using the $ p $-ELDB $ g $ of $ G \cdot H $, we construct a $ p $-$ELDB$ $ f $ on $ G $. Observe that any vertex $ (u_i, v_j) $ in $ G \Box H $ receiving a cost $ l $ dominates the same set of vertices across all factors of $ G $. Consequently, if $ g(u_i, v_j) = l $, define a broadcast $ f $ on $ G $ such that $ f(u_i) = l $. By this construction, $ f $ becomes a $ p $-ELDB on $ G $, contradicting the assumption that $ mcr(G) = k > p $.
\end{proof}
\end{theorem}
Next, we explore by fixing $G$ as $P_m$ and $C_m$ in $G \cdot H$. 
\begin{theorem} \label{pmh}
   $mcr(P_m \cdot H) = 1 $ if and only if $\operatorname{rad}(H)=1 $.
\end{theorem}

\begin{theorem}
    If $m \geq 2$ and $\operatorname{rad}(H) \neq 1$, $mcr(P_m \cdot H) =2 $ and $\gamma_{eb2}(P_m \cdot H) = 2 \ceil* {\frac{m}{5}} $
\end{theorem}

\begin{proof}
Let $V(G)=\{u_1, u_2, \dots u_m\}$ and $V(H)=\{v_1, v_2, \dots v_n\}$. We obtain $2$-$ELDB$ for $P_m \cdot H$ as follows.  By Proposition \ref{prop}, no vertex receives the cost $1$. If a vertex $(u_i,v_j)$ receives cost $2$, it dominates all the vertices of $H^{(u_k)}$, where $k \in\{i-2, i-1, i, i+1, i+2 \}$ and it dominates no other vertices. So, the problem reduces to dominating the vertices of $P_m$ using cost $2$. Since this can be done for any value of $m$ and $P_m$ can be efficiently dominated with the total cost of $2\ceil* {\frac{m}{5}}$, the theorem holds. 

\end{proof}

\begin{coro}
    If $m \geq 2$,
    $ mcr(P_m \cdot H)=\begin{cases}
{1} &\mbox{if $\operatorname{rad}(H)=1$ }\\
{2} &\mbox{otherwise}\\

\end{cases} $\\

\end{coro}

It can be noted that if $f$ is a $k$-$ELDB$ of a graph $G$ such that $f(u) \neq 1, \forall u \in V(G)$, then $f$ restricted to $G^{(v_1)}$ is a $k$-$ELDB$ of $G \cdot H$. We call a graph $G$, a $1$-free if it possesses a $k$-$ELDB$ $f$, such that $f(u) \neq 1, \forall u \in V(G)$. For example, $P_n$ are $1$-free  graphs as any $P_n$ can be efficiently dominated with cost $2$ alone. We call $P_n$ is $[2]-\{1\}$ efficiently dominatable. Every graph $G$ is $[\operatorname{rad}]-\{1,2, \dots , \operatorname{rad}-1\}$ efficiently dominatable.  

\begin{theorem}
   If a graph $G$ is $[k]-\{1\}$-efficiently dominatable,  then $G \cdot H $ is $[k]-\{1\}$-efficiently dominatable
\end{theorem}

\begin{coro} \label{clex}
     If $k$ is the minimum value for which a $1$-free graph $G$ is $[k]-\{1\}$ efficiently dominatable,  then $mcr(G \cdot H) = k$.  
\end{coro}

\begin{coro}
    If a graph $G$ is $1$-efficiently dominatable, then $mcr(G \cdot H) \geq 2 $ if and only if $\operatorname{rad}(H) \neq 1$
\end{coro}

\begin{theorem}
$mcr(C_m \cdot H) = 1$ if and only if $m \equiv 0 (mod 3)$ and $\operatorname{rad}(H)=1$.\\ And in all other cases,\\
$mcr(C_m \cdot H)=\begin{cases}
 
{2} &\mbox{if $m \equiv 0(mod 5) $ and $m\in\{3,4\}$  }\\
{4} &\mbox{if $m\in\{9,16,18,23$\}}\\
{5} &\mbox{if $m=11$}\\
{6} &\mbox{if $m=13$}\\
{3} &\mbox{otherwise}\\
\end{cases} $\\
\end{theorem}

% \begin{theorem}

% if $\operatorname{rad}(H) \neq 1$,
% $mcr(C_m \cdot H)=\begin{cases}
 
% {2} &\mbox{if $m \equiv 0(mod 5) $ and $m\in\{3,4\}$  }\\
% {4} &\mbox{if $m\in\{9,16,18,23$\}}\\
% {5} &\mbox{if $m=11$}\\
% {6} &\mbox{if $m=13$}\\
% {3} &\mbox{otherwise}\\

% \end{cases} $\\

% \end{theorem}

\begin{proof}

    Since no vertex receives cost $1$, and vertices receiving cost $2,3,4,5,6$ dominates $5,7,9,11,13$ number of vertices on $C_m$, the problem can be viewed as a number theory question: write $m$ as repeated summation of $5,7,9,11,13$ such that use smaller numbers as much as possible. Hence, the theorem holds.

\end{proof}
% In the next section, we prove that determining whether a graph $G$ is $[k]-\{1\}$ efficiently dominatable or not is NP-Complete.

\subsubsection{Strong Product of graphs}
\begin{definition} \cite{imrich2000}\label{def_str}
    A Strong product of two graphs $G$ and $H$ is denoted by $G\boxtimes H$, whose vertex set is $V(G) \times V(H)$, in which two vertices $(u,v)$ and $(u', v')$ are adjacent if \\\\
(i) $u=u'$ and $vv' \in E(H)$, or $v=v'$ and $uu' \in E(G)$ \\
(ii) $uu' \in E(G) $ and $vv' \in E(H)$.

\end{definition}
Throughout this section, we use $u_i's$ and $v_j's$ for vertices of $G$ and $H$, respectively. $G^{(v_j)}$ and $H^{(u_i)}$ are defined analogously to that defined in the above section of the lexicographic product of graphs.

\begin{theorem}\cite{imrich2000}\label{dis_str}
The distance betweeen two vertices $(u,v)$ and $(u',v')$ in the graph $G\boxtimes H$ is $max\{d(u,u'),d(v,v')\}$.
\end{theorem}

\begin{coro}\label{dom_str}
    If $f(u_i,v_j)=d$ and if $d_H(v_l, v_j) \leq d$, then $(u_i,v_j)$ dominates vertex $(u_p,v_j)$ if and only if it dominates the vertex $(u_p, v_l)$.
\end{coro}

\begin{theorem}\label{mcr_str}
    $mcr(G\boxtimes H)\geq max\{mcr(G), mcr(H)\}$
\end{theorem}
\begin{proof}
Without loss of generality, assume $ mcr(G) \geq mcr(H) $. Let $ mcr(G) = k $. Assume the contrary, that is $ mcr(G \boxtimes H) = l < k $. Consider the factor $ G^{(v_1)} $, corresponding to a vertex $ v_1 $ of $ H $. In $ G \boxtimes H $, if a vertex $ (u_i, v_j), j \neq 1 $  receives a cost $ d $ and dominates set of vertices $S$ of $ G^{(v_1)} $, assign the same cost $ d $ to $ (u_i, v_1) $ and $0$ to $ (u_i, v_j)$. By Corollary \ref{dom_str}, $ (u_i, v_1) $ dominates the same set $S$ of vertices of $ G^{(v_1)}$. Repeat this process for every broadcasting vertex $ (u_i, v_j), j \neq 1 $, that dominate one or more vertices of $ G^{(v_1)} $. This construction yields an $ l $-$ELDB$ for $ G^{(v_1)} $, contradicting the assumption that $ mcr(G) = k > l $.
\end{proof}

\begin{theorem} \label{ed_str} \cite{abay2009}
If $G$ and $H$ are efficiently dominatable, then $G \boxtimes H$ is efficiently dominatable.
\end{theorem}

That is, in theorem \ref{ed_str}, it was proved that if $mcr(G)=1$ and $mcr(H)=1$, then $mcr(G\boxtimes H)=1$. That is the equality of theorem \ref{mcr_str} holds in this case. Next, we consider $C_m \boxtimes P_n$.
\begin{theorem}\label{cmpnbroad}
    $mcr(C_m \boxtimes P_n) = mcr(C_m)$ 
\end{theorem}

\begin{proof} 

Let $ mcr(C_m) = k $. We prove this by constructing an efficient $ k $-limited dominating broadcast of $ C_m \boxtimes P_n $. 

Note that $ P_n $ admits a perfect $ d $-code for all $ d \leq \operatorname{rad}(P_n) $. 

\noindent \textbf{Case(i):}  $ m \equiv 0 \pmod{3} $\\
The theorem follows directly from Theorems \ref{Cm} and \ref{ed_str}.  

\noindent \textbf{Case(ii):}$ m = 4 $\\
Assign a perfect  $ 2 $-code of $ P_n $ to one of the $ P_n $ factors. By Corollary \ref{dom_str}, every vertex in $ C_4 \boxtimes P_n $ is dominated. Thus, $ mcr(C_4 \boxtimes P_n) \leq 2 $.  

\noindent \textbf{Case(iii):} $ m = 7 $\\
Assign a perfect $ 3 $-code of $ P_n $ to one of the $ P_n $ factors. Again, by Corollary \ref{dom_str}, every vertex in $ C_7 \boxtimes P_n $ is dominated. Hence, $ mcr(C_7 \boxtimes P_n) \leq 3 $.  

\noindent \textbf{Case(iv):} For $ m $ values other than those mentioned above\\
Assign a perfect $ 2 $-code of $ P_n $ to one of the $ P_n $ factors, say $ P_n^{(v_3)} $. This assignment dominates all vertices in $ C_m^{(v_j)} $ for $ 1 \leq j \leq 5 $. The induced subgraph $ H $ of undominated vertices forms $ P_{m-5} \boxtimes P_n $. Assign weights of 2 and 1 to the vertices of $ H $ such that each vertex with weight 2 dominates 5 vertices, and each vertex with weight 1 dominates 3 vertices, within their respective $ P_{m-5} $ factors of $ H $. Consequently, $ mcr(C_m \boxtimes P_n) \leq mcr(C_m) $. By Theorem \ref{mcr_str}, it follows that $ mcr(C_m \boxtimes P_n) = mcr(C_m) $.
\end{proof}

\begin{coro}
    The broadcast defined in the proof of theorem \ref{cmpnbroad} is $\gamma_{ebk}-$broadcast of $C_m \boxtimes P_n$
\end{coro}

\begin{coro}
   If a graph $ H $ has perfect $1$-code and a perfect $2$-code, then $mcr(C_m \boxtimes H) = mcr(C_m)$.
\end{coro}

\begin{theorem}
 If $\operatorname{rad}(G)=1$, then $mcr(G \boxtimes H) = mcr(H)$, \\ $\gamma_{ebk}(G \boxtimes H )= \gamma_{ebk}(H)$, where $k=mcr(H)$
\end{theorem}

\begin{proof}

Consider the $k$-$ELDB$ of $H$ with respect to the factor $H^{(u_i)}$, where $u_i$ represents the central vertex of $G$. If a broadcasting vertex $v_j$ in $H$ dominates a vertex $v_p$ in $H$, then the vertex $(u_i, v_j)$, inheriting the same cost as $v_j$, dominates all vertices in $G^{(v_p)}$. Hence, the theorem holds.   
\end{proof}
\begin{coro}
 $mcr(K_{1,n} \boxtimes H )= mcr(K_{m} \boxtimes H) = mcr(H)$, $\gamma_{ebk}(K_{1,n} \boxtimes H )= \gamma_{ebk}(K_m \boxtimes H)=\gamma_{ebk}(H)$, where $k=mcr(H)$.
\end{coro}

Thus far, given a graph $ G $, we have determined $ mcr(G) $. In the next section, we address the converse problem.

\section{Computational hardness of efficient $k$-limited broadcast domination problem}
Given any integer $k$, does there exist a graph $G$ such that $mcr(G)=k$?
We answer the question of the existence of such graphs by construction.

\subsection{Construction of a class of trees $T_k$}
A vertex adjacent to a leaf vertex is said to be a support vertex.

\begin{obs}\label{obs_oTk}
If $f$ is an efficient dominating broadcast on a graph $G$, then there exists no support vertex at a distance $f(v)$ from a broadcasting vertex $v$.
\end{obs}

For $k\geq 2$, we give a recursive construction for a family of trees $T_k$ with $mcr(T_k)=k$. That is, there exists no efficient $(k-1)$-limited dominating broadcast on $T_k$.\\
	\begin{construction}{\label{Tk}}

\noindent Assigning step: Let $T_1 =K_2$ \\
Iterative step: Hook two copies of $P_3$, each at the diametrical endpoints of the graph $T_{i-1}$ to get $T_i$\\
Recursive step: Repeat Step 2 until $i=k$\\

\end{construction}
		\begin{figure}[!h]
			\centering

   \includegraphics[scale=0.77]{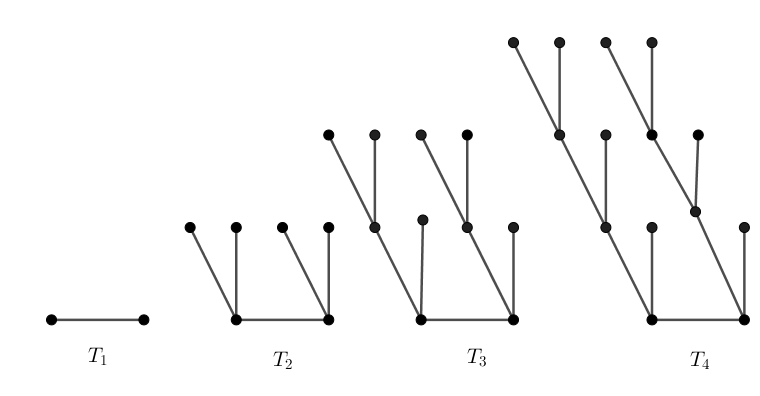}
			\caption{Graphs of $T_k$}
			\label{Tk}
		\end{figure}

\noindent We observe that $T_k$ is a bicentral tree.
\begin{lemma}\label{Tklemma}
    Graph $T_k$ admits an unique efficient $k$-limited dominating broadcast, where one of the central vertex receives the cost $k$. i.e.,  $mcr(T_k)=k$. 
\end{lemma}

\begin{proof}
    Let $ f $ be the assignment such that $ f(u) = k $ for $ u $ being one of the central vertices of the graph. Suppose there exists another assignment $ g $ that differs from $ f $. According to Observation \ref{obs_oTk},  the leaf vertices excluding the diametrical endpoints of $ T_k $ has to be assigned a cost of 1. This assignment ensures dominating all vertices, except for the two end copies of $ P_3 $. Since eccentricity $ ecc(v) \geq k+1 $ for all vertices $ v $ other than the central vertices, the assignment $ f $ is the unique efficient $ k $-limited dominating broadcast up to isomorphism.
\end{proof}

\subsection{NP-Completeness of efficient $k$-limited broadcast domination}

The decision version of $k$-$ELDB$ problem is stated as follows\\
\textbf{Instance:} \textit{A graph $G$ and an integer $k$}. \\
\textbf{Decision Problem:} Determine whether $G$ admits an efficient $k$- limited broadcast domination.\\
To prove that the $k$-$ELDB$ problem is NP-commplete, we establish its reduction from Exact 3-SAT problem as below.
The EXACT 3-SAT problem:\\
\textbf{Instance:} \textit{A CNF formula with $m$ clauses $C_1,C_2,\dots, C_m$ and $n$ literals $u_1,u_2,\dots,u_n$}.\\
\textbf{Decision Problem:} Determine whether there exists a truth assignment function $T:u_i\rightarrow \{0,1\}$ such that the CNF formula is satisfied with exactly one of the three literals in each clause being assigned the value true(1).\\

In the next theorem, we give a reduction from EXACT 3-SAT to $k$-efficient broadcast domination problem. 
\begin{theorem} \label{npc}
    Efficient $k$-limited broadcast domination problem is NP-complete.
\end{theorem}
\begin{proof}
    We present a reduction from the EXACT 3-SAT problem. Given an instance of the Exact 3-SAT problem, which consists of CNF formula with $m$ clauses, $C_1, C_2,...C_m$ and $n$ literals, $U=\{u_1,u_2,...,u_n\}$. We construct a graph of
    $(3k-2)m + (4k-2)n$ number of vertices and certain number of edges as shown in the construction.\\
For each $u_i$, draw a copy of $T_k$, and name the central vertices as $u_i$ and $\overline{u_i}$. Draw m vertices $C_1,C_2, \dots C_m$. If $C_i = u_j \vee \overline{u_k} \vee u_l$, then draw three edges incident from $C_i$ to $u_j, \overline{u_k}$ and $ u_l$. Subdivide each of these edge $k-1$ times, which lead to a path from $C_i$ to $u_j$, $C_i$ to $\overline{u_k}$, $C_i$ to $u_l$. We call them as $Q_{ij}, Q_{i\overline{k}}, Q_{il}$ respectively. Let $P_{ij} = Q_{ij}-\{u_j\}$, $P_{i\overline{k}} = Q_{i\overline{k}}- \{ \overline{u_k} \} $.

		\begin{figure}[!h]
			\centering
			\includegraphics[scale=0.77]{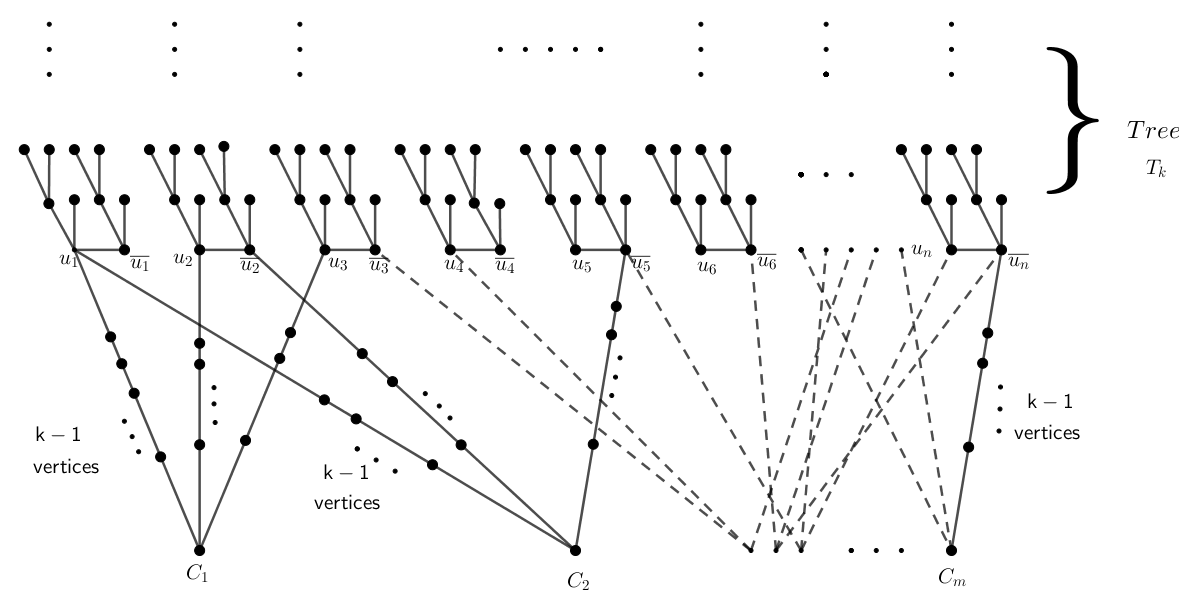}
			\caption{Reduction graph $G_1$ for an example of CNF, where $C_1=\{u_1,u_2,u_3\}, C_2=\{u_1,\overline{u_2},\overline{u_5}\}$, etc...}
			\label{k-eds}
		\end{figure}

Claim 1: If there exists a truth assignment $T:u_i\rightarrow \{0,1\}$ satisfying CNF for the EXACT 3-SAT problem, then there exists an efficient $k$-limited dominating broadcast $f$ on a graph $G_1$. 
%%\begin{proof}
%%If $T(u_i) = 1$, then assign $f(u_i)=k$ \\
%%$f(v)=0$ for the remaining vertices of $G_1$.
%%Each vertex in $T_k$ is dominated by the corresponding central vertex($u_i$ or $\overline{u_i}$, which has received the cost $k$. Since only one literal is true for each clause $C_j$, the vertex corresponding to that literal dominates $C_j$ and $d(C_j, u_i) > k$ for all other broadcasting vertices. Since the vertices of degree two are at most a distance $k$ from corresponding $u_i$ or $\overline{u_i}$, and at least a distance $k+1$ from other broadcasting vertices, they are dominated exactly once.
%%Therefore $f$ is an efficient $k$-limited dominating broadcast.
%%\end{proof}

Proof of claim 1:
    If $ T(u_i) = 1 $, then assign $ f(u_i) = k $.  
If $ T(u_i) = 0 $, then assign $ f(\overline{u_i}) = k $.  
Set $ f(v) = 0 $ for the remaining vertices of $ G_1 $.

Each vertex in $ T_k $ is dominated by the corresponding central vertex ($ u_i $ or $ \overline{u_i} $) that has been assigned the cost $ k $. Since only one literal is true for each clause $ C_j $, the vertex corresponding to that literal dominates $ C_j $, and $ d(C_j, u_i) > k $ for all other broadcasting vertices. 

The vertices of degree two are at most a distance $ k $ from the corresponding $ u_i $ or $ \overline{u_i} $, and at least a distance $ k+1 $ from other broadcasting vertices, ensuring they are dominated exactly once. 

Therefore, $ f $ is an efficient $ k $-limited dominating broadcast.

If we prove the $f$ obtained in claim 1 is the only efficient $k$-limited dominating broadcast on $G_1$, then the other way is also true. In claim 2, we provide a proof of this assertion.

%$u_i$ joining $C_j$ and $u_i$ or $C_j$ and $\overline{u_i}$ is conne All the $k-1$ subdivided vertices are dominated by one of the central vertex  which are between the path joining $C_j$'s and $u_i$

Claim 2: If there exists an efficient $k$-limited dominating broadcast on $G_1$, then there exists a truth assignment $T:u_i\rightarrow\{0,1\}$  that satisfies CNF.

\noindent Proof of claim 2:
    Based on Observations \ref{obs_oTk} and Lemma \ref{Tklemma}, in any efficient $k$-limited dominating broadcast function $f$, the vertices in $T_k$ must be dominated by either the corresponding $u_i$ or $\overline{u_i}$, and no vertex in $P_{ij}$ can dominate vertices in $T_k$. Consequently, we have $f(u_i) = k$ or $f(\overline{u_i}) = k$. It is important to note that all vertices in $P_{ij}$, except for $C_i$, are dominated. Since these vertices are at most a distance $k$ from $u_j$ and $\overline{u_j}$, one of $u_j$ or $\overline{u_j}$ must receive the cost $k$.

Without loss of generality, let $C_i = u_j \vee \overline{u_k} \vee u_l$. All three literals are at a distance $k$ from $C_i$. Therefore, exactly one of them (without loss of generality, $u_j$) will receive the cost $k$, while the other two ($\overline{u_k}$ and $u_l$) will not. However, for each $u_p$, where $1 \leq p \leq n$, either $u_p$ or $\overline{u_p}$ must receive the cost $k$. Consequently, $\overline{u_l}$ and $u_k$ will receive the cost $k$.

Thus, for every clause $C_i$, all three literals within that clause are at a distance $k$ from $C_i$. To dominate $C_i$ exactly once, exactly one of these three literals receives the cost $k$, implying that the negated literals of the other two receives the cost $k$. Therefore, the broadcasting vertices are either $u_i$ or $\overline{u_i}$.

Let this assignment be denoted as $f$. We then construct the corresponding truth assignment $T$. If $f(u_i) = k$, then $T(u_i) = 1$. If $f(\overline{u_i}) = k$, then $T(u_i) = 0$. 

Suppose $f$ provides an efficient $k$-limited dominating broadcast. In that case, every clause vertex $C_i$ is dominated exactly once, implying that one of its three literals is true. Hence, $T$ is a satisfying assignment for the EXACT 3-SAT problem.

\end{proof}
We proved that for general graphs, determining whether $mcr(G)=k$ is NP-Complete.

\section{Conclusion and future work}

The open problem posed in the broadcast domination chapter of Structures of Domination in graphs (\cite{haynes2021structures})—determining the minimum value of $ k $ for which a graph admits an efficient $ k $-limited dominating broadcast—serves as the primary motivation for this research. In this work, we constructed examples of graphs that are not $ k $-efficiently dominatable and established that determining whether a general graph is $ k $-efficiently dominatable is NP-complete. Furthermore, we investigated this concept on standard graphs, lexicographic products, and strong products of graphs by computing $ \text{mcr}(G) $ and $ \gamma_{ebk}(G) $, alongside deriving other results. As a direction for future research, we are currently exploring efficient $k$-limited dominating broadcasts on trees and other classes of graphs.

% Suppose there exists an efficient dominating broadcast $f$ such that $f(v)=1$ for all the broadcasting vertices, 
% then the set of all broadcasting vertices ($V_{f}^+$) represents an efficient dominating set. 

% \backmatter
% \section*{Statements and Declarations}

% \bmhead{Funding} No funding has been provided for this research.
% \bmhead{Competing Interests} The authors have no competing interests to disclose.

% \noindent
% If any of the sections are not relevant to your manuscript, please include the heading and write `Not applicable' for that section. 

\backmatter

%\bmhead{Acknowledgements}

%\section*{Declarations}

%%===========================================================================================%%
%% If you are submitting to one of the Nature Portfolio journals, using the eJP submission   %%
%% system, please include the references within the manuscript file itself. You may do this  %%
%% by copying the reference list from your .bbl file, paste it into the main manuscript .tex %%
%% file, and delete the associated \verb+\bibliography+ commands.                            %%
%%===========================================================================================%%

\bibliography{m} 
%\bibliographystyle{sn-mathphys-num}
% common bib file
%% if required, the content of .bbl file can be included here once bbl is generated
%%\input sn-article.bbl

\end{document}